\documentclass[11pt,oneside,english]{amsart}
\usepackage[latin9]{inputenc}
\usepackage{babel}
\usepackage{amsthm}
\usepackage{amssymb}
\usepackage{graphicx}
\usepackage[unicode=true,pdfusetitle,
 bookmarks=true,bookmarksnumbered=true,bookmarksopen=true,bookmarksopenlevel=1,
 breaklinks=false,pdfborder={0 0 1},backref=false,colorlinks=false]
 {hyperref}

\makeatletter

\newcommand*\LyXZeroWidthSpace{\hspace{0pt}}

\numberwithin{equation}{section}
\numberwithin{figure}{section}
\theoremstyle{plain}
\newtheorem*{thm*}{\protect\theoremname}
\theoremstyle{plain}
\newtheorem{thm}{\protect\theoremname}
\theoremstyle{plain}
\newtheorem{conjecture}[thm]{\protect\conjecturename}
\theoremstyle{definition}
\newtheorem{defn}[thm]{\protect\definitionname}
\theoremstyle{remark}
\newtheorem*{rem*}{\protect\remarkname}
\theoremstyle{plain}
\newtheorem{lem}[thm]{\protect\lemmaname}
\theoremstyle{plain}
\newtheorem{prop}[thm]{\protect\propositionname}
\theoremstyle{remark}
\newtheorem{rem}[thm]{\protect\remarkname}

\makeatother

\providecommand{\conjecturename}{Conjecture}
\providecommand{\definitionname}{Definition}
\providecommand{\lemmaname}{Lemma}
\providecommand{\propositionname}{Proposition}
\providecommand{\remarkname}{Remark}
\providecommand{\theoremname}{Theorem}

\begin{document}
\title{\noindent Local Unknottedness of Planar Lagrangians with Boundary}
\maketitle
\begin{center}
Zi-Xuan Wang
\par\end{center}

\begin{center}
Uppsala University
\par\end{center}

\noindent \newpage{}

\section*{Abstract}

\noindent We show the smooth version of the nearby Lagrangian conjecture
for the 2-dimensional pair of pants and the Hamiltonian version for
the cylinder. In other words, for any closed exact Lagrangian submanifold
of $T^{*}M$, there is a smooth or Hamiltonian isotopy, when $M$
is a pair of pants or a cylinder respectively, from it to the 0-section.
For the cylinder we modify a result of G. Dimitroglou Rizell for certain
Lagrangian tori to show that it gives the Hamiltonian isotopy for
a Lagrangian cylinder. For the pair of pants, we first study some
results from pseudo-holomorphic curve theory and the planar Lagrangian
in $T^{*}\mathbb{R}^{2}$, then finally using a parameter construction
to obtain a smooth isotopy for the pair of pants. 

\noindent \newpage{}

\section*{Acknowledgement}

\noindent First of all, I would like to express my gratitude to my
supervisor Georgios Dimitroglou Rizell for his countless help throughout
this project, from studying preliminaries, suggesting the topic, answering
my sometimes boundless questions, to consummate the details. This
is my first time writing a mathematical research paper and thanks
to him, it has really been a pleasant and improving experience. Also,
I really appreciate the help from the instructors of all courses that
I have taken. You have really broadened my horizon and deepened my
understanding in mathematics. Finally, thank my friends, my family
and Gästrike-Hälsinge Nation for the support, either psychological
or financial, in various formats. 

\noindent \newpage{}

\noindent \tableofcontents{}

\noindent \newpage{}

\section{Introduction}

\subsection{Background}

\noindent A smooth manifold $\left(W,\omega\right)$ is called a symplectic
manifold if it is equipped with a closed and non-degenerate 2-form
$\omega$. It has to be of even-dimension because of the non-degeneracy
of $\omega$. For an n-dimensional smooth manifold $M$, its cotangent
bundle $\left(T^{*}M,-d\theta\right)$, where $\theta=\sum_{i=1}^{n}p_{i}dq_{i}$
is the tautological 1-form, is naturally a symplectic manifold.

\noindent Lagrangian submanifolds are smooth half-dimensional submanifolds
of symplectic manifolds on which the symplectic form vanishes. For
a cotangent bundle $\left(T^{*}M,-d\theta\right)$, its section $\left(m,\alpha\right)$
is Lagrangian if the 1-form $\alpha$ satisfies $d\alpha=0$. Moreover,
if $\alpha$ is exact, the section is called an exact Lagrangian.
Lagrangian submanifolds have been shown to exhibit many rigidity phenomena
since Gromov\textquoteright s pseudoholomorphic curve theory, such
as the result by Eliashberg-Polterovich in \cite{key-10} numbered
as Theorem 9 here
\begin{thm*}
\noindent Any flat at infinity Lagrangian embedding of $\mathbb{R}^{2}$
into the standard symplectic $\mathbb{R}^{4}$ is isotopic to the
flat embedding via an ambient compactly supported smooth isotopy of
$\mathbb{R}^{4}$.
\end{thm*}
\noindent In the smooth category, an isotopy is a family of diffeomorphisms
$\rho_{t}$ such that $\rho_{0}=Id$. This family is generated by
a family of vector fields $\left\{ X_{t}\right\} $ s.t. 
\[
\frac{d}{dt}\rho_{t}=X_{t}\left(\rho_{t}\right)
\]
If each $X_{t}$ is Hamiltonian, i.e. there exists a smooth family
of functions $H_{t}:T^{*}M\rightarrow\mathbb{R}$, called Hamiltonian
functions s.t. 
\[
\iota\left(X_{t}\right)\omega=dH_{t}
\]
this isotopy is called a Hamiltonian isotopy. 

\noindent Eliashberg-Polterovich considered the case of a Lagrangian
disc, and my work here concerns a generalization: the case of a Lagrangian
pair of pants. The main theorem of this article can be stated as 
\begin{thm*}
\noindent Any flat outside a compact set Lagrangian embedding of $\mathbb{R}^{2}-\left\{ 0,1\right\} $
into the standard symplectic cotangent bundle of the same manifold
is isotopic to the flat embedding via an ambient compactly supported
smooth isotopy of 
\[
T^{*}\left(\mathbb{R}^{2}-\left\{ 0,1\right\} \right)=\mathbb{R}^{2}\times\left(\mathbb{R}^{2}-\left\{ 0,1\right\} \right)=\mathbb{R}^{4}-\left(\mathbb{R}^{2}\times\left\{ 0\right\} \cup\mathbb{R}^{2}\times\left\{ 1\right\} \right)
\]
\end{thm*}
\noindent This result is a direct consequence of Theorem 18, whose
formulation is adapted to the strategy of the proof. 

\noindent The same methods are expected to show that the analogous
result holds for $T^{*}\left(\mathbb{R}^{2}-\left\{ p_{1},\cdots,p_{m}\right\} \right)$,
i.e. the complement of m points. However, the special motivation behind
studying pair of pants is that a closed surface of genus $g\geq2$
admits a pair of pants decomposition. If one can show, via stretching
the neck, that an exact Lagrangian surface inside a cotangent bundle
is isotopic to pieces in a pair of pants decomposition, where the
pieces have standard behavior near their boundaries, then this plus
the result for the pair of pants implies that any exact Lagrangian
surface is smoothly isotopic to the zero section. 

\noindent This can be seen as a strategy to partially prove the following
conjecture from 1986 due to V.I. Arnol'd:
\begin{conjecture}
\noindent (The nearby Lagrangian conjecture) Let M be a closed manifold.
Any closed exact Lagrangian submanifold of $T^{*}M$ is Hamiltonian
isotopic to the 0-section. 
\end{conjecture}

\noindent By demanding that the exact Lagrangian agrees with $M=\mathbb{R}^{2}$
outside a compact subset, the situation is similar to that of a closed
manifold. So do $M=\mathbb{R}^{n}$. So far, the conjecture has only
been established in the cases when $M=\mathbb{R}^{1},S^{1},\mathbb{R}^{2},\mathbb{R}^{1}\times S^{1},\mathbb{T}^{2},S^{2}$.
Conjecture 1 for $M=\mathbb{R}^{2}$ can be rephrased as ``local
planar Lagrangians in $T^{*}\mathbb{R}^{2}$ are unknotted'', and
our goal can be rephrased as ``local planar Lagrangians in $T^{*}(\mathbb{R}^{2}-\left\{ 0,1\right\} )$
are unknotted''. The result for $\mathbb{R}^{2}$ in the smooth category
is proved by Eliashberg-Polterovich in \cite{key-10}, whose methods
provide great enlightenment for our proof, and the Hamiltonian version
is proved by the same authors in another article in \cite{key-7}.
In this article, we also establish the nearby Lagrangian conjecture
in the case of the cylinder $M=\mathbb{R}\times S^{1}$ by adapting
the proof of the Hamiltonian classification result for Lagrangian
tori in $T^{*}\mathbb{T}^{2}$ from \cite[Theorem B]{key-9}.
\begin{thm*}
\noindent (Theorem 17, the nearby Lagrangian conjecture for the cylinder)
Let $L\subset T^{*}\mathbb{T}^{2}$ be an exact Lagrangian torus which
coincides with the zero section above the subset $\theta_{2}\in\left[-\delta,\delta\right]$.
Then $L$ is Hamiltonian isotopic to the zero section by a Hamiltonian
isotopy which is supported in the same subset. 
\end{thm*}

\subsection{Organization of this paper}

\noindent Firstly in section 2, we will present and prove some preliminaries
from Gromov's pseudo-holomorphic curve theory. Then in section 3,
we will go through the proof of local smooth unknottedness of planar
Lagrangians. In section 4, we will explain how Theorem B, which seems
to be the torus version of nearby Lagrangian conjecture in \cite{key-9}
gives an isotopy for the cylinder, and finally in section 5 prove
the local unknottedness of planar Lagrangians with two punctures in
the smooth category. 

\section{Preliminaries from Pseudo-Holomorphic Curve Theory}

\subsection{Moser's trick}

\noindent In \cite{key-1}, J. Moser invented the following method
which turns out to be useful in many cases: Consider two k-forms $\alpha_{1}$
and $\alpha_{0}$ on a smooth manifold $M$ and one wants to find
a diffeomorphism 
\[
\varphi:M\rightarrow M,\varphi^{*}\alpha_{1}=\alpha_{0}.
\]
Moser's idea is to find a family of diffeomorphisms $\varphi_{t},0\leq t\leq1$
for a family of forms $\alpha_{t}$ connecting $\alpha_{1}$ and $\alpha_{0}$
s.t.
\[
\varphi_{t}^{*}\alpha_{t}=\alpha_{0}.
\]
This is in fact
\[
0=\frac{d}{dt}\phi_{t}^{*}\alpha_{t}=\phi_{t}^{*}\left(\alpha_{t}^{\prime}+\mathcal{L}_{X_{t}}\alpha_{t}\right)
\]
Use Cartan's formula, 
\[
\alpha_{t}^{\prime}+d\iota_{X_{t}}\alpha_{t}+\iota_{X_{t}}d\alpha_{t}=0
\]
In our case $\alpha_{t}=\omega_{t}$ is a symplectic form, so it's
closed. The equation becomes
\[
\omega_{t}^{\prime}+d\iota_{X_{t}}\omega_{t}=0.
\]

\noindent As an application, we have a standard result
\begin{thm}
\noindent Let $\omega_{0}$ and $\omega_{1}$ be two symplectic forms
on a compact manifold $M$ that belong to the same de Rham cohomology
class, and $\omega_{t}:=\omega_{0}+td\beta=\left(1-t\right)\omega_{0}+t\omega_{1}$
is symplectic for each $t\in\left[0,1\right]$ Then there is a symplectomorphism
$\phi:\left(M,\omega_{0}\right)\rightarrow\left(M,\omega_{1}\right)$. 
\end{thm}

\begin{proof}
\noindent There exists a 1-form $\beta$ s.t. $\omega_{1}=\omega_{0}+d\beta$.
Let $\omega_{t}:=\omega_{0}+td\beta=\left(1-t\right)\omega_{0}+t\omega_{1}$.
Then the equation
\[
\omega_{t}^{\prime}+d\iota_{X_{t}}\omega_{t}=0
\]
becomes
\[
L.H.S.=d(\beta+\iota_{X_{t}}\omega_{t})=0
\]
which is solvable. 
\end{proof}

\subsection{Gromov's pseudo-holomorphic curve theory}

\subsubsection{Almost complex structure}
\begin{defn}
\noindent An almost complex structure $J\in End\left(TM\right)$ on
a symplectic manifold $\left(M,\omega\right)$ is a smooth linear
structure $J_{m}$ on each tangent space $T_{m}M$ which satisfies
$J_{m}^{2}=-Id$. $J$ is $tamed$ by the symplectic form $\omega$
if
\[
\omega\left(v,Jv\right)>0,\forall v\neq0.
\]
$J$ is compatible with $\omega$ if $\omega\left(v,Jv\right)$ is
a Riemannian metric. A manifold that admits an almost complex structure
is called an almost complex manifold. 
\end{defn}

\begin{rem*}
\noindent Every symplectic manifold admits compatible almost complex
structures. The space of tamed and compatible almost complex structures
are denoted by 
\[
J^{tame}\left(M,\omega\right)
\]
 
\[
J^{comp}\left(M,\omega\right)
\]
. 
\end{rem*}
\begin{lem}
\noindent (Gromov 1985, \cite{key-11}) The spaces $J^{tame}\left(M,\omega\right)$
and $J^{comp}\left(M,\omega\right)$ are both contractible. 
\end{lem}

\begin{proof}
\noindent Sketch of proof: The key is to prove that for each tangent
space the lemma holds. One can identify the latter with a convex and
open subset of the space of metrics on a vector space, hence contractible.
A similar but more complicated identification proves the former.
\end{proof}
\noindent A compatible almost complex structure on a manifold $M$
is a section over $M$ of the fibre bundle whose fibres on each point
are compatible almost complex actions of that point. Sometimes it's
hard to construct a compatible structure over the whole symplectic
manifold, and the above Lemma gives us a way to firstly construct
a section on a suitable submanifold of $M$, and then extend this
section over the entire manifold $M$. This extension is ensured by
the contractibility of fibres. We will apply this to $M=\mathbb{CP}^{2}$
in the next section. 

\subsubsection{Positivity of intersection}
\begin{defn}
\noindent Let $\left(\Sigma,j\right)$ be a Riemann surface. A map
from this Riemann surface to an almost complex manifold $u:\left(\Sigma,j\right)\rightarrow\left(M,J\right)$
is said to be $J-holomorphic$ (also called pseudo-holomorphic) if
it satisfies the fully non-linear first order PDE
\[
\bar{\partial}_{J}u=\frac{1}{2}\left(du+J\circ du\circ j\right)=0
\]
of Cauchy-Riemann type. When $\Sigma=\mathbb{CP}^{1}$, $u$ is called
a J-holomorphic sphere.
\end{defn}

\noindent From the classical intersection theory, we have
\begin{prop}
\noindent (McDuff 1994, \cite{key-16}) Consider a connected holomorphic
curve $u:\Sigma\rightarrow M$ and a holomorphic hypersurface $D\subset M$,
i.e. the complex dimension of $D$ = the complex dimension of $M$
minus one, such that u is not contained inside D. Then:

\noindent $\bullet$ $u$ and $D$ intersect in a discrete subset;

\noindent $\bullet$ each geometric intersection point gives a positive
contribution to the algebraic intersection number $\left[u\right]\cdot\left[D\right]\geq0$;

\noindent $\bullet$ if an intersection point moreover is not a transverse
intersection (e.g. a tangency or an intersection of $D$ and a singular
point of $u$), then that geometric point contributes at least +2.
\end{prop}

\noindent This is of special importance to dim-4 manifolds, because
under this occasion the hypersurface $D$ is also two-dimensional.
We have
\begin{thm}
\noindent (McDuff 1994, \cite{key-16}) Positivity of intersection
in dim 4: two closed distinct J -holomorphic curves $u$ and $u^{\prime}$
in an almost complex 4-manifold $(M,J)$ have only a finite number
of intersection points. Each such point $x$ contributes a number
$k_{x}\geq1$ to the algebraic intersection number $\left[u\right]\cdot\left[u^{\prime}\right]$;.
Moreover, $k_{x}=1$ iff the curves $u$ and $u^{\prime}$ intersect
transversally at $x$.
\end{thm}

\subsubsection{Existence of a pseudo-holomorphic line passing through two points}

\noindent Recall that in algebraic geomrtry, there exists precisely
one algebraic curve of degree one (i.e. homologus to $L\in H_{2}\left(\mathbb{CP}^{n}\right)=\mathbb{Z}\cdot L$)
that passes through two given points $P_{1}\neq P_{2}\in\mathbb{CP}^{n}$:
the complex line
\[
\mathbb{CP}^{1}\rightarrow\mathbb{CP}^{n},\left[x_{1}:x_{2}\right]\mapsto x_{1}\cdot P_{1}+x_{2}\cdot P_{2}
\]
unique up to reparameterization. 

\noindent For pseudo-holomorphic spheres, we have
\begin{thm}
\noindent (Gromov 1985, \cite{key-11}) There exists a unique up to
reparameterization holomorphic curve of degree one (i.e. homologus
to $L\in H_{2}\left(\mathbb{CP}^{n}\right)=\mathbb{Z}\cdot L$) that
passes through two given points $P_{1}\neq P_{2}\in\mathbb{CP}^{n}$:
the complex line
\[
\mathbb{CP}^{1}\rightarrow\mathbb{CP}^{n},\left[x_{1}:x_{2}\right]\mapsto x_{1}\cdot P_{1}+x_{2}\cdot P_{2}.
\]
\end{thm}

\begin{proof}
\noindent If there exists a curve 
\[
u:\left(\Sigma,j\right)\rightarrow\left(\mathbb{CP}^{n},J_{0}\right)
\]
 in class $\left[u\right]=L$ is not of the above form, then we can
find a linear hyperplane (denoted by $H$) $\mathbb{CP}^{n-1}\subset\mathbb{CP}^{n}$
which is tangent to the curve at some point but not contain it. Positivity
of intersection of the curve and the hyperplane implies that $H\cdot L\geq2$,
because a tangency contributes at least 2 to the intersection number.
However, $H\cdot\left[u\right]=1$, contradiction. 
\end{proof}
\noindent This will be of great importance when n=2. We will discuss
it later in section 3.

\section{Local Unknottedness of Planar Lagrangians without Boundary}

\noindent In this section we will elaborate Eliashberg-Polterovich's
proof of the smooth version of the nearby Lagrangian conjecture of
$\mathbb{R}^{2}$, which will provide great insight for our case about
pair of pants. 
\begin{thm}
\noindent (Eliashberg-Polterovich 1993, \cite{key-10}) Any flat at
infinity Lagrangian embedding of $\mathbb{R}^{2}$ into the standard
symplectic $\mathbb{R}^{4}$ is isotopic to the flat embedding via
an ambient compactly supported smooth isotopy of $\mathbb{R}^{4}$.
\end{thm}

\begin{rem}
\noindent In fact, this isotopy can be made Hamiltonian by constructing
a family of suitable hypersurfaces. See \cite{key-7}. 
\end{rem}

\noindent Consider the standard linear space $\mathbb{R}^{4}$ with
coordinates $\left(p_{1},q_{1},p_{2},q_{2}\right)$ such that the
symplectic form $\omega=dp_{1}\wedge dq_{1}+dp_{2}\wedge dq_{2}$.
Let $l=\left\{ p_{1}=p_{2}=0\right\} $ be a Lagrangian plane and
let $L\subset\mathbb{R}^{4}$ be a Lagrangian submanifold which is
diffeomorphic to $\mathbb{R}^{2}$ and which coincides with $l$ outside
a compact subset. We have to prove that $L$ is isotopic to $l$ via
a compactly supported isotopy of $\mathbb{R}^{4}$. To do this, we
will modify $\omega$ to make $l$ and $L$ symplectic at the same
time, so that we can use compactification to make them projective
lines and construct a smooth isotopy. 
\begin{lem}
\noindent Consider the family of symplectic forms $\omega_{\varepsilon}=\omega+\varepsilon dq_{1}\land dq_{2}$,
$\varepsilon>0$ small. There exists a compactly supported isotopy
of $\mathbb{R}^{4}$ that coincides with a linear symplectic plane
outside of a compact subset which takes $L$ to an $\omega_{\varepsilon}-$symplectic
surface. 
\end{lem}

\begin{proof}
\noindent On a closed tubular neighborhood $V$ of $L$, one can find
a closed 2-form $\tau$ such that $\exists$ a compactly supported
1-form $\lambda$ on $V$, $d\lambda=\tau-dq_{1}\land dq_{2}$, which
coincides with $dq_{1}\land dq_{2}$ outside a compact subset of $V$
and makes $L$ $\tau-$symplectic. Choose a bump function $\rho$
on $\mathbb{R}^{4}$ which vanishes outside $V$ and equals 1 near
$L$. Set
\[
\omega_{\varepsilon}^{\prime}=\omega+\varepsilon\left(dq_{1}\land dq_{2}+d\left(\rho\lambda\right)\right)
\]
then $L$ is $\omega_{\varepsilon}^{\prime}$-symplectic. The difference
$\omega_{\varepsilon}^{\prime}-\omega_{\varepsilon}=\varepsilon d\left(\rho\lambda\right)$,
which is 0 outside a compact subset of $\mathbb{R}^{4}$, and exact
inside it. So one can use Moser's linear construction to conclude
that $\omega_{\varepsilon}$ and $\omega_{\varepsilon}^{\prime}$
are isotopic. This isotopy is compactly supported and takes $L$ into
an $\omega_{\varepsilon}-$symplectic surface, denoted by $L^{\prime}.$ 
\end{proof}
\begin{lem}
\noindent There exists a symplectic form on $\mathbb{R}^{4}$ which
tames $J_{0}$, coincides with $\omega_{\varepsilon}$ on some given
subset, and with a multiple of the Fubini-Study metric outside a compact
subset. 
\end{lem}

\begin{proof}
\noindent Take an interpolation between $\parallel z\parallel^{2}$
and $C\cdot log\left(1+\parallel z\parallel^{2}\right),$ for $C\gg0$
sufficiently large. Take $\frac{i}{2}\partial\overline{\partial}$
of this function. If the function is strictly convex as a function
of the radius, then the obtained 2-form, which is an interpolation
between $\omega_{\varepsilon}$ applied in the compact subset where
the above isotopy takes place, and the Fubini-Study form 
\[
\omega_{FS}=C\cdot\frac{i}{2}\partial\overline{\partial}\left(log\left(1+\parallel z\parallel^{2}\right)\right),
\]
far away, is symplectic. Since $J_{0}$ is tamed by $\omega_{\varepsilon}$
and $\omega_{FS}$ at the same time, $J_{0}$ is tamed by the interpolated
2-form, too. Actually, recall that any symplectic form of the form
$\frac{i}{2}\partial\overline{\partial}f$ for a smooth real function
$f$ is a Kähler form with respect to the standard complex structure,
and hence the standard complex structure is compatible with this symplectic
form. 
\end{proof}
\noindent After the above interpolation, the symplectic volume of
$\mathbb{R}^{4}$ becomes finite, so that one can use a projective
line $\Gamma_{\infty}$ to replace remote areas in $\mathbb{R}^{4}$
where the Fubini-Study form is applied, and make $\mathbb{R}^{4}$
into $\mathbb{CP}^{2}$. For convenience, we simply call the above
process by ``compactifying $\mathbb{R}^{4}$ into $\mathbb{CP}^{2}$''. 
\begin{lem}
\noindent One can compactify $\left(\mathbb{R}^{4},\omega_{\varepsilon}\right)$
into $\mathbb{CP}^{2}$ by adding a line $\Gamma_{\infty}$ at infinity,
while making the symplectic plane $l=\left\{ \left(p_{1},p_{2}\right)=\left(0,0\right)\right\} $
a projectice line $\gamma_{0}$, $L$ an symplectic embedded sphere
$\Sigma$, which is homologus to $\mathbb{CP}^{1}$ and intersects
the infinity line $\Gamma_{\infty}$ at the same point $P$ as $\Gamma_{\infty}\cap\gamma_{0}$.
\end{lem}

\begin{proof}
\noindent Due to our choice of the bump function $\rho$ vanishing
outside a tubular neibourhood $V$, $L^{\prime}$ agrees with $L$
at infinity (thus with the standard plane $l$ too) when applying
the isotopy. This allows us to, as in Lemma 12, compactify $\left(\mathbb{R}^{4},\omega_{\varepsilon}\right)$
with coordinates 
\[
\left(\mathbb{R}^{4}=\left(p_{1},q_{1},p_{2},q_{2}\right),\omega_{\varepsilon}=\left(dp_{1}-\frac{\varepsilon}{2}dq_{2}\right)\wedge dq_{1}+\left(dp_{2}+\frac{\varepsilon}{2}dq_{1}\right)\wedge dq_{2}\right)
\]
 into $\mathbb{CP}^{2}$ with coordiantes $\left[z_{1},z_{2}\right]$,
namely $z_{1}\sim\left(x_{1},y_{1}\right)\sim\left(p_{1}-\frac{\varepsilon}{2}q_{2},q_{1}\right),z_{2}\sim\left(x_{2},y_{2}\right)\sim\left(\frac{2}{\varepsilon}p_{2}+q_{1},\frac{\varepsilon}{2}q_{2}\right)$
and make use of the results of pseudo holomorphic curves in $\mathbb{CP}^{2}$.
Choose an almost complex structure $j$ (which is also complex here
in $\mathbb{R}^{4}$), such that $\omega_{\varepsilon}\left(\cdot,j\cdot\right)$
is a Euclidian metric. Take a large ball $B$ such that $L^{\prime}$
coinsides with $l$ outside it, and compactify the ball to $\mathbb{CP}^{2}$
by adding a line $\Gamma_{\infty}$ at infinity. Denote the compactifications
of $\omega_{\varepsilon}$ and $j$ by $\Omega$ and $J_{0}$ respectively.
After such a compactication the symplectic line $l$ corresponds to
a projective line $\gamma_{0}=\left\{ z_{1}=iz_{2}\right\} $ and
our knot $L$ is compactified to an symplectic embedded sphere $\Sigma$,
which is homologus to $\mathbb{CP}^{1}$ and intersects $\Gamma_{\infty}$
at the same point $P$ as $\Gamma_{\infty}\cap\gamma_{0}$. See Figure
5,2.1.
\end{proof}
\noindent In order to prove the theorem, it is enough to show that:
\begin{thm}
\noindent $\Sigma$ is smoothly isotopic to the projective line $\gamma_{0}$
defined in Lemma 13 via an isotopy of $\mathbb{CP}^{2}$ which fixes
$\Gamma_{\infty}$. Moreover, this isotopy can be taken to fix a neighborhood
of $\Gamma_{\infty}$.
\end{thm}

\begin{rem}
\noindent Given a compatible $J$, a $J-$holomorphic line on $\mathbb{CP}^{2}$
is an embedded 2-sphere $C\subset\mathbb{CP}^{2}$ which is homologous
to $\mathbb{CP}^{1}$ and whose tangent space $T_{x}C$ is $J-$invariant
for all points $x\in C$. Under this definition, Theorem 8 is to say
that for each compatible almost complex structure $J$ on $\mathbb{CP}^{2}$
and for each two distinct points $A,B\in\mathbb{CP}^{2}$ there exists
a unique $J-holomorphic$ line which passes through $A$ and $B$.
Moreover this line depends smoothly on $J,A,B$.

\noindent This implies that given a point $x_{0}\in\mathbb{CP}^{2}$,
there is a pencil of lines based on this point, i.e. a set of pseudo-holomorphic
lines that all intersect at $x_{0}$. This pencil of lines forms a
foliation away from a point, where each leaf is determined by its
tangency at the point $x_{0}$. 
\end{rem}

\begin{proof}
\noindent 
\begin{figure}
\includegraphics{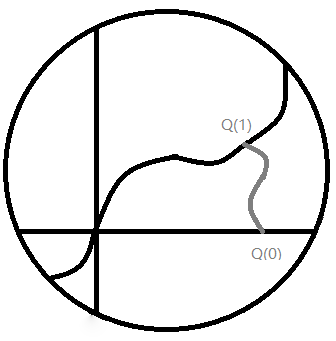}

\caption{Fig.2 in \cite{key-10}. The vertical line is $\Gamma_{\infty}$,
horizonal lines are $\gamma_{0}$ and $\Sigma$. Gray line is the
path $Q\left(t\right)$.}

\end{figure}

\noindent \LyXZeroWidthSpace{}

\noindent (See Figure 5,2.1) Choose a smooth path $Q\left(t\right),t\in\left[0,1\right]$
on $\mathbb{CP}^{2}-\Gamma_{\infty}$ such that $Q\left(0\right)\in\gamma_{0}$
and $Q\left(1\right)\in\Sigma$. There exists a smooth family of compatible
almost complex structures $J\left(t\right),t\in\left[0,1\right]$
such that 

\noindent $J\left(0\right)=J_{0}$

\noindent $\Sigma$ is $J\left(1\right)$-holomorphic 

\noindent $\Gamma_{\infty}$ is $J\left(t\right)$-holomorphic for
all t.

\noindent For the existence, one can first choose some appropriate
almost complex structure along $\Gamma_{\infty}$, or say a section
on the bundle of compatible actions over $\Gamma_{\infty}$. Since
the space of compatible almost complex structures forms a contractible
space, we can extend this section to $\mathbb{CP}^{2}$, which gives
$J\left(t\right)$. 

\noindent By Theorem 8, there exists a unique $J\left(t\right)-$holomorphic
line, denoted by $C\left(t\right)$ passing through $P$ and $Q\left(t\right)$
changing smoothly w.r.t. parameter $t$. The deformation $\left\{ C\left(t\right),t\in\left[0,1\right]\right\} $
gives an isotopy between $\gamma_{0}$ and $\Sigma$. 

\noindent Before we finally extend $\left\{ C\left(t\right),t\in\left[0,1\right]\right\} $
to an ambient isotopy of $\mathbb{CP}^{2}$ which preserves $\Gamma_{\infty}$,
it remains to show that $C\left(t\right)$ intersects $\Gamma_{\infty}$
at the unique point $P$ transversally. If not,namely $C\left(t\right)$
intersect $\Gamma_{\infty}$ at more than one points(counted by multiplicity),
by positivity of intersection in dim 4, each of them contributes positively
to the intersection number $C\left(t\right)\cdot\Gamma_{\infty}$,
so the number is bigger than one. However, $C\left(0\right)\cdot\Gamma_{\infty}=1$,
contradicts the fact that the intersection number should remain unchanged
for homologus lines.

\noindent Additionally, after a smooth deformation of $Q\left(t\right)$,
one can require that for all $t\in\left[0,1\right]$, $Q\left(t\right)\cap\gamma_{0}$
not only at the point $x_{o}$, but also in a neighborhood $I\subset\gamma_{0}$
of $x_{0}$. Here it's important that $Q\left(t\right)$ is transverse
to $\Gamma_{\infty}$. Then one can make the isotopy intersect $I\times\Gamma_{\infty}$
only at $I$, thus obtain a compactly supported isotopy. 
\end{proof}

\section{Local Unknottedness of One-punctured Planar Lagrangians}

\subsection{The Nearby Lagrangian conjecture for the cylinder}

\noindent We first explain the statement of \cite[Theorem B]{key-9}
and show that this is actually equivalent to the nearby Lagrangian
conjecture for the cylinder. 

\noindent By cutting a torus $\mathbb{T}^{2}=\left(S^{1}\times S^{1}\right)$
parameterized by $\theta_{1}$ and $\theta_{2}$ between $\left\{ \theta_{2}=s\right\} $
and $\left\{ \theta_{2}=t\right\} $, we get two cylinders. This motivates
us to realize that a Lagrangian cylinder inside the cotangent bundle
of $T^{*}\left(S^{1}\times I\right)$ which is standard near the boundary
can be extended to a Lagrangian torus in $T^{*}\left(S^{1}\times S^{1}\right)$
which is equal to the zero section in ${\theta_{2}\in\left[-\delta,\delta\right]}$
in the base. Namely, this isotopy of a Lagrangian torus keeps $S^{1}\times e^{i\left[-\delta,\delta\right]}$
fixed, so it is actually an isotopy of the rest part of a torus, which
is a cylinder. We start by recalling the following result. 
\begin{thm}
\noindent (\cite[Theorem B(2)]{key-9}, 2019) Suppose that $L\subset(T^{*}\mathbb{T}^{2},d\lambda_{\mathbb{T}^{2}})$
is an exact Lagrangian embedding. Then for any $\theta\in S^{1}$
consider the properly embedded Lagrangian disc with one interior point
removed
\[
\dot{D}_{\mathbf{P}^{0}}\left(\theta\right):=\left(S^{1}\times\left\{ \theta\right\} \right)\times\left(\left\{ p_{1}^{0}\right\} \times(-\infty,p_{2}^{0}]\right)\subset\mathbb{T}^{2}\times\mathbb{R}^{2}=T^{*}\mathbb{T}^{2},
\]
\[
\mathbf{p}^{0}:=\left(p_{1}^{0},p_{2}^{0}\right)
\]
If it is the case that
\[
L\cap\dot{D}_{\mathbf{p}^{0}}\left(e^{is}\right)=\partial\dot{D}_{\mathbf{p}^{0}}\left(e^{is}\right)=S^{1}\times\left\{ e^{is}\right\} \times\left\{ \mathbf{p}^{0}\right\} 
\]
holds for all $\mid s\mid<\epsilon$, then the Hamiltonian isotopy
can be assumed to be supported outside of the subset
\[
\bigcup_{\mid s\mid<\delta}\dot{D}_{\mathbf{p}^{0}}\left(e^{is}\right):=S^{1}\times e^{i\left[-\delta,\delta\right]}\times\left\{ p_{1}^{0}\right\} \times(-\infty,p_{2}^{0}]
\]
for some $0<\delta<\epsilon$ sufficiently small (note that for symplectic
action reasons, we may not be able to Hamiltonian isotope the Lagrangian
to the constant section $\mathbb{T}^{2}\times\left\{ \mathbf{p}^{0}\right\} $).
\end{thm}

\noindent 
\noindent The above result can in particular be applied to the torus
which is an extension of a Lagrangian cylinder in $T^{*}\left(S^{1}\times I\right)$
which is standard above the boundary. However, we need to strengthen
it in the following manner: make sure that the Hamiltonian isotopy
of the torus is fixed above the entire subset $\theta_{2}\in\left[-\delta,\delta\right]$. 

\noindent We prove the following strengthening of the result, 
\begin{thm}
\noindent (The nearby Lagrangian conjecture for the cylinder) Let
$L\subset T^{*}\mathbb{T}^{2}$ be an exact Lagrangian torus which
coincides with the zero section above the subset $\theta_{2}\in\left[-\delta,\delta\right]$.
Then $L$ is Hamiltonian isotopic to the zero section by a Hamiltonian
isotopy which is supported in the same subset. 
\end{thm}

\begin{proof}
\noindent We follow exactly the same steps as the proof of theorem
B in \cite[Section 9]{key-9}. 

\noindent We prove by constructing a solid torus with core removed
which is foliated by pseudo-holomorphic punctured discs. 

\noindent The main step of the proof is to construct a proper embedding
of a solid torus $\dot{\mathcal{T}}\subset T^{*}T^{2}$ with its core
removed, which is foliated by pseudoholomorphic discs, and whose boundary
is the Lagrangian $L$. For technical reasons this solid torus with
core removed is constructed as a solid torus $\mathcal{T}\subset\mathbb{CP}^{1}\times\mathbb{CP}^{1}$
foliated by pseudoholomorphic discs, which gives rise to $\dot{\mathcal{T}}$
after removing the four holomorphic lines $\left(\left\{ 0,\infty\right\} \times\mathbb{CP}^{1}\right)\cup\left(\mathbb{CP}^{1}\times\left\{ 0,\infty\right\} \right)$.
Recall that the complement of this divisor is a neighbourhood of the
zero section of $T^{*}T^{2}$. In order to obtain the solid torus
with the sought properties it is crucial that we have a family of
(punctured) pseudoholomorphic discs that we can start with. Namely,
one starts from the family of punctured pseudo-holomorphic discs 
\[
C\left(e^{is}\right)=S^{1}\times\left\{ e^{is}\right\} \times\left(-\infty,p_{1}^{0}\right)\times\left\{ p_{2}^{0}\right\} \subset T^{*}\mathbb{T}^{2}\setminus L
\]
with boundary on $L$ that exists by the assumption that $L$ is standard
above the neighborhood $\theta_{2}\in\left(-\delta,\delta\right)$.
We proceed to explain how the additional assumptions made in Theorem
16 here give us additional control over the solid torus, as compared
to the assumptions in Theorem 15 above which was proven in \cite{key-9}.
The goal is to use positivity of intersection to say that $\mathcal{\dot{T}}$
is standard above the region $\theta_{2}\in\left(-\delta,\delta\right)$. 

\noindent In \cite[Theorem B(2)]{key-9}, the condition about the
intersection with the Lagrangian disc $\dot{D}_{\mathbf{P}^{0}}\left(\theta\right)$,
says that the Hamiltonian isotopy is ``one-sided'' fixed near some
Lagrangian disc, in the sense that the subset $\cup_{\mid s\mid<\delta}\dot{D}_{\mathbf{P}^{0}}\left(e^{is}\right)$
contains only $p_{2}<p_{2}^{0}$, but not $p_{2}>p_{2}^{0}$ and $p_{2}=p_{2}^{0}$.
As in Section 9.1 of \cite{key-9}, denote a smooth one-dimensional
family of embedded symplectic cylinders by 
\[
C\left(e^{is}\right)=S^{1}\times\left\{ e^{is}\right\} \times\left(-\infty,p_{1}^{0}\right)\times\left\{ p_{2}^{0}\right\} \subset T^{*}\mathbb{T}^{2}\setminus L
\]
where $L$ is the Lagrangian torus, and for $p_{2}<p_{2}^{0}$, $p_{2}>p_{2}^{0}$,
there are cylinders 
\[
C\left(p_{2},e^{is}\right)=S^{1}\times\left\{ e^{is}\right\} \times\mathbb{R}\times\left\{ p_{2}^{0}\right\} \subset T^{*}\mathbb{T}^{2}\setminus L.
\]
Together, we can find a well-defined compatible almost complex structure
$J$ on $\left(T^{*}\mathbb{T}^{2}\setminus L,d\lambda_{\mathbb{T}^{2}}\right)$
which is cylindrical outside of a compact subset, and which agrees
with $J_{cyl}$ in a neighborhood of the union of cylinders. In \cite{key-9},
only $p_{2}<p_{2}^{0}$ was considered, however, if we look at Figure
5,2.16 in \cite{key-9}, there is also space for cylinders corresponds
to $p_{2}>p_{2}^{0}$ above the $p_{2}$ plane. This property is a
direct consequence of the assumption that $L$ coincides with the
zero section above $\theta_{2}\in\left[-\delta,\delta\right]$.

\noindent Let us consider the family of cyliders
\[
C\left(e^{is}\right)\cup C\left(p_{2},e^{is}\right),\mid s\mid\leq\delta,p_{2}<p_{2}^{0},p_{2}>p_{2}^{0}.
\]

\noindent 
\begin{figure}

\includegraphics[scale=0.8]{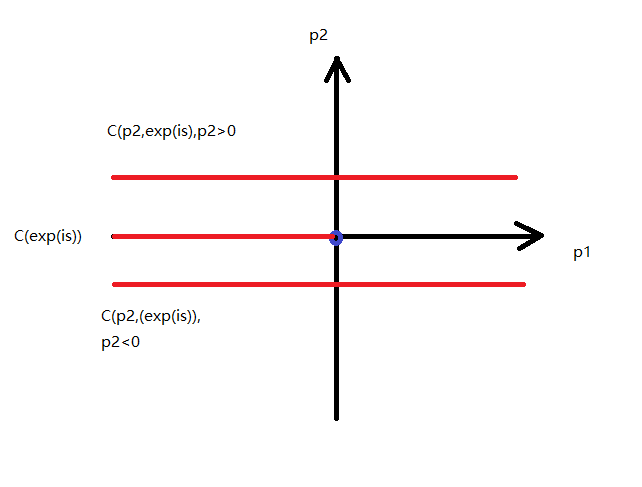}\caption{Figure 16 in \cite{key-9}, where the blue origin point is the only
intersection, red lines are leaves of $\mathcal{\dot{T}}$.}

\end{figure}

\noindent \LyXZeroWidthSpace{}

\noindent The argument in \cite[Section 9]{key-9} produces an embeded
solid torus $\mathcal{\dot{T}}$ inside $T^{*}\mathbb{T}^{2}$ with
core removed, whose boundary is $L$, and which is foliated by pseudo-holomorphic
cylinders with boundary on $L$. Part of these cylinders are given
by the explicitly constructed standard cylinders $C\left(e^{is}\right)$.
After compactifying $T^{*}\mathbb{T}^{2}$ to $S^{2}\times S^{2}$,
$\mathcal{\dot{T}}$ becomes a solid torus. The positivity of intersection
argument in Lemma 9.8(2) of \cite{key-9} which shows that the solid
torus is disjoint from $C\left(p_{2},e^{is}\right)$ also shows that
the interior of these solid tori (which are foliated by pseudo-holomorphic
curves) are disjoint from the cylinders $C\left(p_{2},e^{is}\right)$.
In particular this solid torus coinsides with the union of standard
cylinders $C\left(e^{is}\right)$ at origin but not in $p_{1}>0,p_{2}=0$
inside the subset $\theta_{2}\in\left[-\delta,\delta\right]$, because
if so, either it is contained inside the domain, or it intersects
the domain at a discrete subset. The former is impossible, by compactness,
and the latter is also impossible, because it will also intersect
the nearby cylinders $C\left(p_{2},e^{is}\right)$ by continuity.
Hence $\mathcal{\dot{T}}$, and the isotopy, intersects 
\[
C\left(e^{is}\right)\cup C\left(p_{2},e^{is}\right),\mid s\mid\leq4\delta,p_{2}<p_{2}^{0},p_{2}>p_{2}^{0}
\]
 exactly at the origin(s) of the Figure, in other words, we get an
isotopy of the cylinder which is supported in the cotangent bundle
of the cylinder. 
\end{proof}
\noindent \newpage{}

\section{Local Unknottedness of Two-punctured Planes}

\noindent While the analog of Theorem 9 for one-punctured planes (i.e.
cylinder) is shown in Section 4.1, namely any Lagrangian inside $T^{*}\left(\mathbb{R}^{2}-\left\{ 0\right\} \right)$
which coincides with the zero-section outside of a compact subset
is Hamiltonian isotopic to the zero section, the analogous result
for $T^{*}\left(\mathbb{R}^{2}-\left\{ 0,1\right\} \right)$ (i.e.
pair of pants) is unknown. Rather than Hamiltonian isotopy, we will
try to prove the smooth version in this section. 

\noindent Observe that here ``compact'' requires the isotopy to
be fixed not only near $\infty$, but also 0 and 1. So we add back
0-fibre and 1-fibre, use the same idea of modification of $\omega$
to make the Lagrangian symplectic and compactification as the non-punctured
case to construct an isotopy fixed at infinity. However, more needs
to be done to make this isotopy fixed near 0 and 1. As an analog of
Lemma 11 and 12, we have
\begin{lem}
\noindent There exists a diffeomorphism of $T^{*}\mathbb{R}^{2}=\mathbb{R}^{4}$
which takes the zero section to the complex line $\gamma_{0}$, the
Lagrangian $L$ to a symplectic surface which coincides with a complex
line $\gamma_{0}$ outside a compact subset, while the fibres over
$0$ and $1$ become mapped to two disjoint symplectic lines $\Gamma_{0}$
and $\Gamma_{1}$, respectively, that intersects $\gamma_{0}$ transversely
and positively in a single point, each are complex near $\gamma_{0}$
and outside a compact subset. 
\end{lem}

\begin{proof}
\noindent First, we perturb $L$ to a symplectic surface $\alpha$
equal to $\gamma_{0}$, (i.e. the zero section, which is a symplectic
linear plane for the symplectic form $\omega_{\varepsilon}$) near
fibres and outside a compact subset. We construct this perturbation
by an application of Moser's trick as in Section 3, on a closed tubular
neighborhood $V$ of $L$, one can find a closed 2-form $\tau$ such
that there exists a 1-form $\lambda$ which is compactly supported
on $V$ and vanishing near the fibres $\sigma_{0}=\left\{ q_{1}=q_{2}=0\right\} ,\sigma_{1}=\left\{ q_{1}=1,q_{2}=0\right\} $
with
\[
d\lambda=\tau-dq_{1}\land dq_{2},
\]
which coincides with $dq_{1}\land dq_{2}$ outside a compact subset
of $V$ and makes $L$ $\tau-$symplectic. Choose a bump function
$\rho$ on $\mathbb{R}^{4}$ which vanishes outside $V$, near the
Lagrangian fibres $\sigma_{0}=\left\{ q_{1}=q_{2}=0\right\} ,\sigma_{1}=\left\{ q_{1}=1,q_{2}=0\right\} $,
and equals 1 near the neighborhood where $L$ is Lagrangian. Set
\[
\widetilde{\omega_{\varepsilon}^{\prime}}=\omega+\varepsilon\left(dq_{1}\land dq_{2}+d\left(\rho\lambda\right)\right).
\]
Then $L$ is $\widetilde{\omega_{\varepsilon}^{\prime}}$-symplectic.
Use Moser's trick to deform $L$ to an $\omega_{\varepsilon}$-symplectic
plane $\alpha$ by a smooth isotopy which is supported in the complement
of $\sigma_{0}$, $\sigma_{1}$ defined above. 

\noindent Second, we perturb Lagrangian fibres to symplectic planes.
This can be done via a family of parallel 2-planes that intersect
$\alpha$ transversely at two points. Namely, in the standard linear
space $\mathbb{R}^{4}$ with coordinates $\left(p_{1},q_{1},p_{2},q_{2}\right)$
such that the symplectic form 
\[
\omega_{\varepsilon}=\left(dp_{1}-\frac{\varepsilon}{2}dq_{2}\right)\wedge dq_{1}+\left(dp_{2}+\frac{\varepsilon}{2}dq_{1}\right)\wedge dq_{2}
\]
Consider a family of parallel pair of 2-planes $\Gamma_{0}^{t}$,
$\Gamma_{1}^{t}$:
\[
\Gamma_{0}^{t}:=\left\{ \left(q_{1},q_{2}\right)^{T}=t\left(\begin{array}{cc}
0 & 1\\
-1 & 0
\end{array}\right)\cdot\left(p_{1},p_{2}\right)^{T}+\left(0,0\right)\right\} 
\]
\[
\Gamma_{1}^{t}:=\left\{ \left(q_{1},q_{2}\right)^{T}=t\left(\begin{array}{cc}
0 & 1\\
-1 & 0
\end{array}\right)\cdot\left(p_{1},p_{2}\right)^{T}+\left(1,0\right)\right\} 
\]
where $\Gamma_{0}^{0}$ is the Lagrangian fibre $\sigma_{0}=\left\{ q_{1}=q_{2}=0\right\} $,
and $\Gamma_{1}^{0}=\sigma_{1}=\left\{ q_{1}=1,q_{2}=0\right\} .$
When $t=0$, $\Gamma_{0}^{0}$, $\Gamma_{1}^{0}$ are a pair of Lagrangian
planes for the symplectic form $\omega_{\varepsilon}$, when $0<t\leq t_{0}$
for some $t_{0}>0$ small, $\Gamma_{0}^{t}$, $\Gamma_{1}^{t}$ are
a pair of symplectic planes for the symplectic form $\omega_{\varepsilon}$.
If needed, one can rescale the $p_{i}$ coordiantes to flatten the
Lagrangian $L$ so that $\Gamma_{0}^{t}$, $\Gamma_{1}^{t}$ will
not intersect $L$ at points other than our desired two transversal
intersection points. 

\noindent Third, map $\gamma_{0}$ to the complex line $\mathbb{C}\times0$
by a linear symplectomorphism. Meanwhile, denote the symplectic images
of $\gamma_{0}$, $\Gamma_{0}^{t_{0}}$, $\Gamma_{1}^{t_{0}}$ by
the same notations. One can obtain a smooth family $\Sigma_{s}$ of
symplectic immersions, where $\Sigma_{0}=\gamma_{0}\cup\Gamma_{0}^{t_{0}}\cup\Gamma_{1}^{t_{0}}$
with fixed intersection points $pt_{0},pt_{1}$ on $\Gamma_{0}^{t_{0}}$,
$\Gamma_{1}^{t_{0}}$ respectively. 

\noindent Finally, use \cite[Proposition 4.9]{key-9} to obtain a
new deformation $\widetilde{\Sigma_{t}}$ through symplectic immersions
with exactly two transverse double points, such that $\widetilde{\Sigma_{0}}=\Sigma_{0}$
and that the deformation fixes $\gamma_{0}$ and the positions $pt_{0}$
and $pt_{1}$ of the double points and where the deformation has support
near the double points and near $\infty$. After such a deformation
we may assume that the sought properties are satisfied.
\end{proof}
\noindent Again as Lemma 12, take an interpolerated symplectic form
and compactify $\mathbb{R}^{4}$ into $\left(\mathbb{CP}^{2},\Omega,J_{0}\right)$
by adding $\Gamma_{\infty}$. The complex lines $\Gamma_{0}$, $\Gamma_{1}$
from the previous proposition become symplectic spheres in the projective
plane. Denote them again by $\Gamma_{0}$, $\Gamma_{1}$. Moreover,
we may assume that this $J_{0}$ makes $\gamma_{0}$$,\Gamma_{0},\Gamma_{1},\Gamma_{\infty}$
simultaneously J-holomorphic. The symplectic plane $\alpha$ becomes
an embedded symplectic sphere, denoted by $\alpha$ again, which intersects
$\Gamma_{\infty}$, $\Gamma_{0}$, $\Gamma_{1}$ at the same points
as $\gamma_{0}$ does. Denote the intersections by $\infty,$ $pt_{0}$,
$pt_{1}$ respectively. See Figure 5,1. In the following we abuse
notation and use $\Gamma_{0}$, $\Gamma_{1}$ in order to denote the
compactifications of the symplectic planes to symplectic degree one
spheres in $\mathbb{CP}^{2}.$

\noindent 
\begin{figure}

\includegraphics{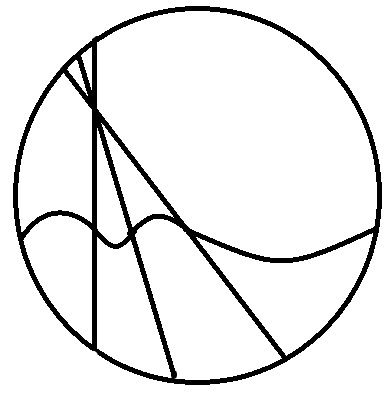}\caption{Vertical lines are $\Gamma_{\infty}$, $\Gamma_{0}$, $\Gamma_{1}$,
they intersect at the same point $x_{0}$. The horizonal line is $\alpha$
which intersect $\Gamma_{\infty}$, $\Gamma_{0}$, $\Gamma_{1}$ at
$\infty,$ $pt_{0}$, $pt_{1}$ respectively. }

\end{figure}

\noindent \LyXZeroWidthSpace{}

\noindent Note the following: $\Gamma_{i}$ are complex planes which
are linear outside a compact subset and near $\mathbb{C}\times0.$
The key point to our proof is that we can interpolate through compatible
almost complex structures that keep $\Gamma_{\infty}$, $\Gamma_{0}$,
$\Gamma_{1}$ J-holomorphic, from one that makes $\alpha$ J-holomorphic,
to one that makes $\gamma_{0}$ J-holomorphic. 
\begin{thm}
\noindent There exists a smooth isotopy of $\mathbb{CP}^{2}$ that
takes $\alpha$ to $\gamma_{0}$ which fixes the line $\Gamma_{\infty}$
pointwise, keeps the intersection point between this isotopy's image
and $\Gamma_{\infty}$, $\Gamma_{0}$, $\Gamma_{1}$ at $\infty,$
$pt_{0}$, $pt_{1}$ respectively, and the image does not intersect
$\Gamma_{\infty}$, $\Gamma_{0}$, $\Gamma_{1}$ at other points.
Moreover, this isotopy can be taken to fix a neighborhood of $\Gamma_{\infty}$,
$\Gamma_{0}$, $\Gamma_{1}$. 
\end{thm}

\begin{proof}
\noindent By Lemma 18, we have an almost complex structure $J_{0}$
making $\gamma_{0},\Gamma_{0},\Gamma_{1},\Gamma_{\infty}$ simultaneously
$J_{0}$-holomorphic. There is also an almost complex structure $J_{1}$
making $\alpha,\Gamma_{0},\Gamma_{1},\Gamma_{\infty}$ simultaneously
$J_{1}$-holomorphic, again as argued in Theorem 14 by extending a
section of a bundle with contractible fibres. Choose the path $J_{t}$
of almost complex structures for which $\Gamma_{0},\Gamma_{1},\Gamma_{\infty}$
remain $J_{t}$-holomorphic for all times t. Consider the unique $J_{t}$-holomorphic
line, denoted by $\alpha_{t}$, passing through the two intersection
points $\infty,0$. We have $\alpha_{0}=\gamma_{0}$, $\alpha_{1}=\alpha$
and the family $\left\{ \alpha_{t},0\leq t\leq1\right\} $ gives the
smooth isotopy from $\alpha$ to $\gamma_{0}$. One can extend this
isotopy to an isotopy of $\mathbb{CP}^{2}$ that fixes $\Gamma_{0},\Gamma_{1},\Gamma_{\infty}$,
by \cite[Proposition 4.8]{key-9}. However, the problem is that $\alpha_{t}$
may not intersect $\Gamma_{1}$ at 1. See Figure 5,2. 

\noindent 
\begin{figure}

\includegraphics{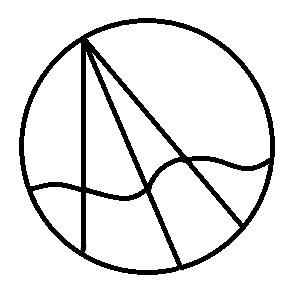}\caption{The intersection points of $\alpha_{t}$ and $\Gamma_{\infty}$, $\Gamma_{0}$
is $\infty$ and $pt_{0}$ respectively, but the intersection of $\alpha_{t}$
and $\Gamma_{1}$may not be $pt_{1}$.}

\end{figure}

\noindent Use parameter $t$ to denote the above isotopy process $f\left(t,x\right),t\in\left[0,1\right],x\in\mathbb{CP}^{2}$
from $\alpha$ to $\gamma_{0}$, with Figure 5,2 as an intermediate
status between them. 

\noindent For every tamed $J$, $\mathbb{CP}^{2}-\left\{ x_{0}\right\} $
can be foliated by $J-$holomorphic spheres that all intersect at
$x_{0}$. Especially, there is a unique $J_{t}$-holomorphic curve
passing through $x_{0}$ and $1$, denoted by $\Gamma^{t}$. There
is a smooth isotopy for each $t\in\left[0,1\right]$ that takes $\Gamma_{1}$
to $\Gamma^{t}$. Parameterize this isotopy by $s$. Together we have
an isotopy 
\[
g\left(t,s,x\right),t\in\left[0,1\right],s\in\left[0,1\right],x\in\mathbb{CP}^{2}
\]
with the property that $g\left(0,s,x\right)=x$, $g\left(1,s,x\right)=x$.
Take the top-path of the two-dimensional domain of $\left(s,t\right)$,
we get the desired smooth isotopy. See Figure 5.3. 

\noindent 
\begin{figure}

\includegraphics[scale=0.8]{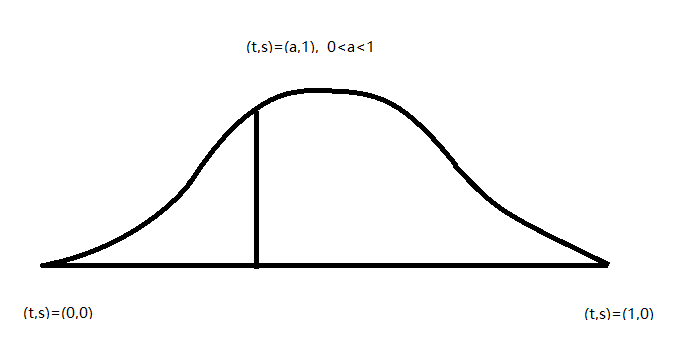}\caption{The two-dimendional domain of the parameter pair $\left(t,s\right)$.}

\end{figure}

\noindent Parameterize this path by $r$, finally we get a smooth
isotopy from $\alpha$ to $\gamma_{0}$, which keeps the intersection
points $\infty,$$\Gamma_{0}^{t_{0}}$, $\Gamma_{1}^{t_{0}}$. By
the same intersection argument as in section 3, this isotopy intersect
$\Gamma_{\infty}$, $\Gamma_{0}$, $\Gamma_{1}$ at no other points. 

\noindent A standard topological argument like that in the end of
Proposition 12 makes the smooth isotopy away from $\Gamma_{\infty}$,
$\Gamma_{0}$, $\Gamma_{1}$. Moreover, one can use the method of
\cite[Theorem 4.6]{key-9} to make the isotopy also symplectic.
\end{proof}

\end{document}